\documentclass[a4paper, 12pt]{article}
\usepackage{}

\usepackage{mathrsfs}
\usepackage{epsfig}
\usepackage{amsmath}
\usepackage{amssymb}
\usepackage{latexsym}
\usepackage{amsfonts}
\usepackage{amsthm}

\usepackage[numbers,sort&compress]{natbib}
\usepackage{graphicx}
\ifx\pdfoutput\undefined  \else
 \fi

\usepackage[centerlast]{caption2}

\usepackage{color}

\usepackage{txfonts}
\usepackage{amssymb}
\usepackage{mathrsfs}
\usepackage{amsfonts}

\newtheorem{theorem}{Theorem}[section]
\newtheorem{lemma}[theorem]{Lemma}

\newtheorem{prop}[theorem]{Proposition}
\newtheorem{conj}[theorem]{Conjecture}
\newtheorem{remark}[theorem]{Remark}

\usepackage{latexsym,amssymb}
\usepackage{amsmath}

\newcommand{\F}{{\mathbb F}}

\begin{document}

\title{Davenport constant of the multiplicative semigroup of the quotient ring $\frac{\F_p[x]}{\langle f(x)\rangle}$}

\author{
Haoli Wang$^1$\thanks{Email:
bjpeuwanghaoli@163.com}  \ \ \ \ \ \
Lizhen Zhang$^{2,3}$\thanks{Corresponding author's email:
lzhzhang0@aliyun.com}  \ \ \ \ \ \
Qinghong Wang$^4$\thanks{Email: wqh1208@aliyun.com} \ \ \ \ \ \
Yongke Qu$^5$\thanks{Email:
yongke1239@163.com}
\\
\\
$^1$ {\small College of Computer and Information Engineering}\\
{\small Tianjin Normal University, Tianjin, 300387, P. R. China}\\
$^{2}${\small Department of Mathematics} \\
{\small Tianjin Polytechnic University, Tianjin, 300387, P. R. China}\\
$^{3}${\small Shanghai Institute of Applied Mathematics and Mechanics} \\
{\small Shanghai University, Shanghai, 200072, P. R. China}\\
$^{4}${\small College of Science}\\
{\small Tianjin University of Technology,
Tianjin, 300384, P. R. China}\\
$^{5}${\small Department of Mathematics}\\
{\small Luoyang Normal University,
Luoyang, 471022, P. R. China}
}

\date{}
\maketitle

\begin{abstract}  Let $\mathcal{S}$ be a finite commutative semigroup. The Davenport constant of $\mathcal{S}$, denoted $D(\mathcal{S})$, is defined to be the least positive integer $d$ such that every sequence $T$ of elements in $\mathcal{S}$ of length at least $d$ contains a subsequence $T'$ with the sum of all terms from $T'$ equaling the sum of all terms from $T$. Let $\F_p[x]$ be a polynomial ring
in one
variable over the prime field $\F_p$, and let $f(x)\in \F_p[x]$. In this paper, we made a study of the Davenport constant of the multiplicative semigroup of the quotient ring $\frac{\F_p[x]}{\langle f(x)\rangle}$. Among other results, we mainly prove that, for any prime $p>2$ and any polynomial $f(x)\in \F_p[x]$ which can be factorized into several pairwise non-associted irreducible polynomials in $\F_p[x]$, then $$D(\mathcal{S}_{f(x)}^p)=D(U(\mathcal{S}_{f(x)}^p)),$$
where $\mathcal{S}_{f(x)}^p$ denotes the multiplicative semigroup of the quotient ring $\frac{\F_p[x]}{\langle  f(x)\rangle}$ and $U(\mathcal{S}_{f(x)}^p)$ denotes the group of units of the semigroup $\mathcal{S}_{f(x)}^p$.
\end{abstract}


\section {Introduction}

Let $G$ be an additive finite abelian group. A sequence $T$ of
elements in $G$ is called a {\sl zero-sum sequence} if the sum of
all terms of $T$ equals to zero, the identify element of $G$. The Davenport constant $D(G)$ of
$G$ is defined to be the smallest integer $d\in \mathbb{N}$ such that,
every sequence $T$ of $d$ elements in $G$ contains a nonempty
subsequence $T'$ with the sum of all terms of $T'$ equaling zero. H. Davenport \cite{Davenport} proposed
the study of this constant in 1965, which aroused a huge variety of further researches (see \cite{GaoGeroldingersurvey} for a survey). Recently, G.Q. Wang and W.D. Gao \cite{wanggao} generalized the Davenport constant to finite commutative semigroups.

 \noindent \textbf{Definition.} \cite{wanggao} \ {\sl Let $\mathcal{S}$ be a finite commutative semigroup. Let $T$ be a sequence of elements in $\mathcal{S}$. We call $T$ a reducible sequence if $T$ contains a proper subsequence $T'$ ($T'\neq T$) such that the sum of all terms of $T'$ equals the sum of all terms of $T$. Define the Davenport constant of the semigroup $\mathcal{S}$, denoted $ D(\mathcal{S})$, the smallest integer $d$ such that every sequence $T$ of length at least $d$ of elements in $\mathcal{S}$ is reducible.}

Moreover, the above two authors together with S.D. Adhikari also made a study of some related additive problems in semigroups (see \cite{AdhikariGaoWang14} and \cite{wang}).
Motivated by their pioneering work on additive problems in semigroups, we study the Davenport constant for the multiplicative semigroup of the quotient ring of a  polynomial ring
in one
variable over the prime field $\F_p$. Our main result is as follows.

\begin{theorem}\label{Theorem irreducible quotient ring}
\ For any prime $p>2$, let $f(x)\in \F_p[x]$ such that $f(x)$ can be factorized into several irreducible polynomials which are not associated each other. Then $$D(\mathcal{S}_{f(x)}^p)=D(U(\mathcal{S}_{f(x)}^p)),$$
where $\mathcal{S}_{f(x)}^p$ denotes the multiplicative semigroup of the quotient ring $\frac{\F_p[x]}{\langle  f(x)\rangle}$.
\end{theorem}

Moreover, in the final concluding section, we conjecture that $$D(\mathcal{S}_{f(x)}^p)=D(U(\mathcal{S}_{f(x)}^p))$$ holds for any prime $p>2$ and any non-constant polynomial $f(x)\in \F_p[x]$, and in particular, we verify it for the case of $f(x)=(x+1)^2$.

\section{The proof of Theorem \ref{Theorem irreducible quotient ring}}

We begin this section by giving some preliminaries.

Let $\mathcal{S}$ be a finite commutative semigroup.
The operation on $\mathcal{S}$ is denoted by $+$.
The identity element of $\mathcal{S}$, denoted $0_{\mathcal{S}}$ (if exists), is the unique element $e$ of
$\mathcal{S}$ such that $e+a=a$ for every $a\in \mathcal{S}$. If $\mathcal{S}$ has an identity element $0_{\mathcal{S}}$, let
$$U(\mathcal{S})=\{a\in \mathcal{S}: a+a'=0_{\mathcal{S}} \mbox{ for some }a'\in \mathcal{S}\}$$ be the group of units
of $\mathcal{S}$.
The zero element of $\mathcal{S}$, denoted
$\infty_{\mathcal{S}}$ (if exists), is the unique element $z$ of $\mathcal{S}$ such that $z+a=z$ for every
$a\in \mathcal{S}$.
Let $$T=x_1x_2\cdot\ldots\cdot x_n=\prod\limits_{x\in \mathcal{S}} x^{\ {\rm v}_x(T)},$$  is a sequence
of elements in the semigroup $\mathcal{S}$, where ${\rm v}_x(T)$ denotes the multiplicity of $x$ in the sequence $T$.
Let $T_1,T_2\in
\mathcal{F}(\mathcal{S})$ be two sequences on $\mathcal{S}$. We call $T_2$
a subsequence of $T_1$ if $${\rm v}_x(T_2)\leq {\rm v}_x(T_1)$$ for every element $x\in \mathcal{S}$, in particular, if $T_2\neq T_1$, we call $T_2$ a {\bf proper} subsequence of $T_1$, and write $$T_3=T_1 \cdot T_2^{-1}$$ to mean the unique subsequence of $T_1$ with $T_2\cdot T_3=T_1$.  By $\lambda$ we denote the
empty sequence.
If $S$ has an identity element $0_{\mathcal{S}}$,  we allow $T=\lambda$ to be empty and adopt the convention
that $\sigma(\lambda)=0_\mathcal{S}$.
We say that $T$ is {\it
reducible} if $\sigma(T')=\sigma(T)$ for some proper subsequence $T'$ of $T$
(Note that, $T'$ is probably the empty sequence $\lambda$ if $\mathcal{S}$
has the identity element $0_{\mathcal{S}}$ and $\sigma(T)=0_{\mathcal{S}}$). Otherwise, we call $T$
{\it irreducible}. For more related terminology used in additive problems in semigroups, one is refereed to \cite{wang}.

\begin{lemma}(\cite{GH}, Lemma 6.1.3) \label{Lemma recusive Davenport constant} \ Let $G$ be a finite abelian group, and let $H$ be a subgroup of $G$. Then, $D(G)\geq D(G/H)+D(H)-1$.
\end{lemma}

For any finite commutative semigroup $\mathcal{S}$ with identity $0_{\mathcal{S}}$, since $U(\mathcal{S})$ is a nonempty subsemigroup of $\mathcal{S}$, we have the following.

\begin{lemma}\label{proposition D(U(G))leq D(G)} (see \cite{wanggao}, Proposition 1.2) \
Let $\mathcal{S}$ be a finite commutative semigroup with identity. Then $D(U(\mathcal{S}))\leq
D(\mathcal{S})$.
\end{lemma}

\begin{lemma}\label{Lemma product of cyclic semigroups} Let $k\geq 1$, and let $n_1,n_2,\ldots,n_k\geq 2$ be positive integers.  Let $\mathcal{S}_i=C_{n_i}\cup \{\infty_i\}$ be a semigroup obtained by a cyclic group of order $n_i$ adjoined with a zero element $\infty_i$ for each $i\in [1,k]$. Let $\mathcal{S}=\mathcal{S}_1\times \mathcal{S}_2\times \cdots \times\mathcal{S}_k$ be the product of $\mathcal{S}_1,\mathcal{S}_2,\ldots,\mathcal{S}_k$. Then $D(\mathcal{S})=D(U(\mathcal{S}))$.
\end{lemma}

 \begin{proof} \ By Lemma \ref{proposition D(U(G))leq D(G)}, we need only to show that
$$D(\mathcal{S})\leq D(U(\mathcal{S})).$$ Observe that $$U(\mathcal{S})=U(\mathcal{S}_1)\times \cdots \times U(\mathcal{S}_k)=C_{n_1}\times \cdots \times C_{n_k}.$$
Let $g_i$ be the generator of the cyclic $C_{n_i}$. Let ${\bf a}=(a_1,a_2,\ldots,a_k)$ be an element of $\mathcal{S}$. We define $$\mathcal{J}({\bf a})=\{i\in [1,k]:a_i=\infty_i\}.$$
We see that for each $i\in [1,k]$, either $a_i=\infty_i$ or $a_i=m_i g_i$ where $m_i\in[1,n_i]$,  and moreover, ${\bf a}\in  U(\mathcal{S})$ if and only if $\mathcal{J}({\bf a})=\emptyset$.

For any index set $I\subseteq [1,k]$, let $\psi_{I}$ be the canonical epimorphism of $\mathcal{S}$ onto the semigroup $\prod\limits_{i\in I}\mathcal{S}_i$ given by
$$\psi_{I}({\bf x})=(x_1',x_2',\ldots,x_{k}')$$ with $$x_i'=0 \mbox{ for } i\in I$$ and $$x_i'=x_i \mbox{ for } i\in [1,k]\setminus I,$$
where ${\bf x}=(x_1,x_2,\ldots,x_k)$ denotes an arbitrary element of $\mathcal{S}$. Note that
$\psi_{I}(\mathcal{S})$ is a subsemigroup of $\mathcal{S}$ which is isomorphic to the semigroup $\prod\limits_{i\in I}\mathcal{S}_i$, the product of the semigroups $\mathcal{S}_i$ with $i\in I$.

Now take an arbitrary sequence $T$ of elements of $\mathcal{S}$ of length at least $U(\mathcal{S})$.
By applying Lemma \ref{Lemma recusive Davenport constant} recursively, we have that
\begin{align}\label{equation |T|geq k+1}
\begin{array}{llll}
|T|&\geq & D(\mathcal{S}) \\
&\geq & D(U(\mathcal{S}))\\
&=& D(\prod\limits_{i\in [1,k]} C_{n_i})\\
&\geq& D(\prod\limits_{i\in [1,k-1]} C_{n_i})+D(C_{n_k})-1\geq D(\prod\limits_{i\in [1,k-1]} C_{n_i})+1\\
&\geq & D(\prod\limits_{i\in [1,k-2]} C_{n_i})+2\\
&\vdots& \\
&\geq & D(C_{n_1})+(k-1)\\
&\geq & 2+(k-1)\\
&=& k+1.\\
\end{array}
\end{align}

It suffices to show that $T$ contains a proper subsequence $T'$ with $\sigma(T')=\sigma(T)$.

Suppose first that all the terms of $T$ are from $U(\mathcal{S})$, i.e., $\mathcal{J}({\bf x})= \emptyset$ for each term ${\bf x}$  of $T$. Since $|T|\geq D(U(\mathcal{S}))$, it follows that $T$ contains a {\bf nonempty} subsequence $V$ with $\sigma(V)=0_{\mathcal{S}}$, i.e., the sum of all terms from $V$ is the identity element of $\mathcal{S}$. This implies that $\sigma(T V^{-1})=\sigma(T V^{-1})+0_{\mathcal{S}}=\sigma(T V^{-1})+\sigma(V)=\sigma(T)$. Then $T'=T V^{-1}$ shall be the required proper subsequence of $T$, we are done.
Hence, we assume that not all the terms of $T$ are from $U(\mathcal{S})$, that is, $$\mathcal{J}(\sigma(T))\neq \emptyset.$$

Note that for each $i\in \mathcal{J}(\sigma(T))$, there exists at least one term of $T$, say ${\bf a_i}$, such that $$i\in \mathcal{J}({\bf a_i}).$$ It follows that there exists
a nonempty subsequence $V$ of $T$ of length at most $|\mathcal{J}(\sigma(T))|$ such that $$\mathcal{J}(\sigma(V))=\mathcal{J}(\sigma(T)).$$
Let $$L=[1,k]\setminus \mathcal{J}(\sigma(T)).$$
Note that $\mathcal{\psi_{L}}(TV^{-1})$ is a sequence of elements in $U(\psi_{L}(\mathcal{S}))\cong U(\prod\limits_{i\in L} \mathcal{S}_i)=\prod\limits_{i\in L} C_{n_i}$. By \eqref{equation |T|geq k+1}, we have that $$|TV^{-1}|\geq 1.$$
By applying Lemma \ref{Lemma recusive Davenport constant} recursively, we have that
$$\begin{array}{llll}
D(U(\psi_{L}(\mathcal{S})))&=&D(\prod\limits_{i\in L} C_{n_i}) \\
&\leq & D(\prod\limits_{i\in [1,k]} C_{n_i})-(k-|L|)\\
&=& D(\prod\limits_{i\in [1,k]} C_{n_i})- |\mathcal{J}(\sigma(T))|\\
&\leq& D(\prod\limits_{i\in [1,k]} C_{n_i})- |V|\\
&=& D(U(\mathcal{S}))- |V|\\
&\leq & |T|- |V|\\
&=& |TV^{-1}|\\
&=& |\psi_{L}(TV^{-1})|.\\
\end{array}$$
It follows that $TV^{-1}$ contains a  {\bf nonempty} subsequence $W$ such that $\sigma(\psi_{L}(W))$ is the identity element of the group $U(\psi_{L}(\mathcal{S}))$, i.e., $$\sigma(\psi_{L}(W))=0_{U(\psi_{L}(\mathcal{S}))}=0_{\mathcal{S}}.$$ Since $V\mid TW^{-1}$, it follows that $\mathcal{J}(\sigma(TW^{-1}))=\mathcal{J}(\sigma(V))=\mathcal{J}(\sigma(T)),$ which implies that
$$\sigma(TW^{-1})+\sigma(\psi_{L}(W))=\sigma(TW^{-1})+\sigma(W).$$
Then we have that
$$\sigma(TW^{-1})=\sigma(TW^{-1})+0_{\mathcal{S}}=\sigma(TW^{-1})+\sigma(\psi_{L}(W))=\sigma(TW^{-1})+\sigma(W)=\sigma(T),$$
and thus,  $T'=TW^{-1}$ is the required proper subsequence of $T$, we are done.
\end{proof}

\bigskip

\noindent {\bf Proof of Theorem \ref{Theorem irreducible quotient ring}} \
Let $$f(x)=f_1(x) f_2(x)\cdots f_k(x)$$ where $f_1(x), f_2(x), \ldots,f_k(x)\in \F_p[x]$ are irreducible and do not associate each other. By the Chinese Remainder Theorem, we have $$\frac{\F_p[x]}{\langle  f(x)\rangle}\cong \frac{\F_p[x]}{\langle  f_1(x)\rangle} \times  \frac{\F_p[x]}{\langle  f_2(x)\rangle} \times \cdots \times \frac{\F_p[x]}{\langle  f_k(x)\rangle}.$$
It follows that the multiplicative semigroup $\mathcal{S}_{f(x)}^p$ of the ring $\frac{\F_p[x]}{\langle  f(x)\rangle}$, is isomorphic to the product of the multiplicative semigroups of $\frac{\F_p[x]}{\langle  f_1(x)\rangle},\frac{\F_p[x]}{\langle  f_2(x)\rangle},\ldots,\frac{\F_p[x]}{\langle  f_k(x)\rangle}$, i.e.,
$$\mathcal{S}_{f(x)}^p\cong \mathcal{S}_{f_1(x)}^p\times \mathcal{S}_{f_2(x)}^p\times \cdots\times \mathcal{S}_{f_k(x)}^p.$$
Since the polynomial $f_i(x)$ is irreducible for each $i\in [1,k]$, we have $\frac{\F_p[x]}{\langle  f_i(x)\rangle}$ is a finite field, and thus, the semigroup $\mathcal{S}_{f_i(x)}^p$ is a cyclic group adjoined with a zero element. Then the conclusion follows from Lemma \ref{Lemma product of cyclic semigroups} immediately.
\qed

\section{Concluding remarks}

In the final section, we propose the further research by suggesting the following conjecture.

\begin{conj}\label{Conjecture 1} \ For any prime $p>2$, let $f(x)\in \F_p[x]$ with ${\rm deg}(f(x))\geq 1$.
Then $$D(\mathcal{S}_{f(x)}^p)=D(U(\mathcal{S}_{f(x)}^p)).$$
\end{conj}

From Theorem \ref{Theorem irreducible quotient ring}, we need only to verify the case that there exists some irreducible polynomial $g(x)\in \F_p[x]$ with $g(x)^2\mid f(x)$. Therefore, we close this paper by making a preliminary verification for example when $f(x)=(x+1)^2$.

\begin{prop}\label{Theorem quotient ring}
\ For a prime $p>2$,   $$D(\mathcal{S}_{(x+1)^2}^p)=D(U(\mathcal{S}_{(x+1)^2}^p)).$$
\end{prop}

\begin{proof} \
In view of Lemma \ref{proposition D(U(G))leq D(G)},
we need only to show that $D(S_{(x+1)^2}^p)\leq
D(U(S_{(x+1)^2}^p)).$ Take an arbitrary sequence $T$ of elements in the semigroup $S_{(x+1)^2}^p)$ with length
\begin{equation}\label{equation length of T}
|T|=D(U(S_{(x+1)^2}^p)).
\end{equation} It suffices to show that $T$ is reducible. Note
that
\begin{equation}\label{equation units}
U(S_{(x+1)^2}^p)=\{ax+b: a,b\in \F_p \mbox{ and }a\neq b\}.
\end{equation}
Let $g$ be a primitive root of the prime $p$. Take the sequence $V={\bf a}_1
{\bf a}_2\cdot\ldots \cdot {\bf a}_{p-1}$ of elements in $S_{(x+1)^2}^p$, where ${\bf a}_1=x$ and ${\bf a}_2= \cdots ={\bf a}_{p-1} =g$. It is easy to check that $V$ is irreducible, which implies
\begin{equation}\label{equation length of davenport}
D(U(S_{(x+1)^2}^p)) \geq |T|+1=p.
\end{equation}

Suppose first that all the terms of $T$ are from
$U(S_{(x+1)^2}^p)$. By \eqref{equation length of T}, we have that $T$ is  reducible, we are done.

Suppose that $T$ contains two terms, say ${\bf a}_1,{\bf a}_2$,
which are not in $U(S_{(x+1)^2}^p)$. By \eqref{equation units}, we have that $(x+1)^2$ divides the product of the two polynomials ${\bf a}_1$ and ${\bf a}_2$, that is,
the sum of the two elements ${\bf a}_1$ and ${\bf a}_2$ is the zero element of the semigroup
$S_{(x+1)^2}^p$, i.e.,
$${\bf a}_1+{\bf a}_2=\infty_{S_{(x+1)^2}^p}.$$
Then $\sigma(T)=\sigma(T\cdot {\bf a}_1^{-1}{\bf a}_2^{-1})+\sigma({\bf a}_1{\bf a}_2)=\sigma(T\cdot {\bf a}_1^{-1}{\bf a}_2^{-1})+\infty_{S_{(x+1)^2}^p}=\infty_{S_{(x+1)^2}^p}=\sigma({\bf a}_1{\bf a}_2)$. Combined with \eqref{equation length of davenport}, we have that ${\bf a}_1{\bf a}_2$ is the required {\bf proper} subsequence of $T$, we are done.

It remains to consider the case that
$T$ contains exactly one element outside the group $U(S_{(x+1)^2}^p)$, say $${\bf a}_1\in S_{(x+1)^2}^p\setminus U(S_{(x+1)^2}^p) \mbox{ and }{\bf a}_i\in U(S_{(x+1)^2}^p) \mbox{ for  }i=2,3,\ldots, |T|.$$
By \eqref{equation units}, we have that
$${\bf a}_1=m(x+1),$$ where
$m\in
\F_p$. We may
assume $T{\bf a}_1^{-1}$ is irreducible (otherwise, $T$ shall be reducible, and we are done). By the definition of irreducible sequences, we have that
$$0_{S_{(x+1)^2}^p}=0_{U(S_{(x+1)^2}^p)}\notin
\sum(T{\bf a}_1^{-1}),$$ where the set $$\sum(T{\bf a}_1^{-1})=\{\sigma(T'): ~T' ~{\rm is~ a ~ nonempty ~ subsequence ~ of }~ T{\bf a}_1^{-1} \}$$ consists of all the elements of the semigroup $S_{(x+1)^2}^p$ that can be represented by the sum of several distinct terms from the sequence $T{\bf a}_1^{-1}$. We see that the sequence $T{\bf a}_1^{-1}$ is a zero-sum free sequence of elements in the group $U(S_{(x+1)^2}^p)$ of length $D(U(S_{(x+1)^2}^p))-1$. It follows that $$\sum(T{\bf a}_1^{-1}) =U(S_{(x+1)^2}^p)\setminus \{0_{U(S_{(x+1)^2}^p)}\},$$
which implies that
there exists a {\bf nonempty} subsequence $W$ of $T{\bf a}_1^{-1}$ with $\sigma(W)=x+2\in \F_p[x]$.
We see that
$$\sigma(W)+{\bf a}_1=(x+2)* m(x+1)=m(x+1)={\bf a}_1.$$ Let $T'=TW^{-1}$. Then we have
that $$\sigma(T')=\sigma(T'{\bf a}_1^{-1})+{\bf a}_1=\sigma(T'{\bf a}_1^{-1})+({\bf a}_1+\sigma(W))=\sigma(T')+\sigma(W)=\sigma(T),$$ and thus, $T'$ is the required proper subsequence of $T$. This completes the proof.
\end{proof}

\bigskip

\noindent {\bf Acknowledgements}

\noindent  This work is supported by NSFC (61303023, 11301381, 11301531, 11371184,  11471244), Science and Technology Development Fund of Tianjin Higher
Institutions (20121003), Doctoral Fund of Tianjin Normal University (52XB1202), NSF of Henan Province (grant no. 142300410304).


\bigskip


\begin{thebibliography}{99}

\bibitem{AdhikariGaoWang14} S.D. Adhikari, W.D. Gao and G.Q. Wang, \emph{Erd\H{o}s-Ginzburg-Ziv theorem for finite commutative semigroups,} Semigroup Forum, \textbf{88} (2014)  555--568.

\bibitem{Davenport} H. Davenport, Proceedings of the Midwestern conference on group theory and number theory. Ohio State University. April 1966.

\bibitem{GaoGeroldingersurvey} W.D. Gao and A. Geroldinger, {Zero-sum problems in finite abelian groups: a survey,} Expo. Math.,  \textbf{24} (2006) 337--369.


\bibitem{GH} A. Geroldinger and F. Halter-Koch, \emph{Non-Unique
Factorizations. Algebraic, Combinatorial and Analytic Theory,}
Pure and Applied Mathematics, vol. 278, Chapman $\&$ Hall/CRC,
2005.

\bibitem{wanggao} G.Q. Wang and W.D. Gao,
\emph{Davenport constant for semigroups,} Semigroup Forum,
\textbf{76} (2008) 234-238.

\bibitem{wang}  G.Q. Wang, \emph{Structure of the largest idempotent-free sequences in finite semigroups,}  arXiv:1405.6278.




\end{thebibliography}
\end{document}